\documentclass[11pt]{amsart}
\usepackage{amssymb, amscd}
\usepackage{latexsym,epsfig}
\usepackage[all]{xy}
\usepackage{pst-all}
\usepackage{pstcol}
\usepackage[normalem]{ulem}

\usepackage[
        hypertexnames=false,
        hyperindex,
        pagebackref,
        pdftex,
        breaklinks=true,
        bookmarks=false,
        colorlinks,
        linkcolor=blue,
        citecolor=red,
        urlcolor=red,
]{hyperref}

\usepackage[mathscr]{eucal}

\numberwithin{equation}{section}
\def\vandaag{\number\day\space\ifcase\month\or
 januari\or februari\or  maart\or  april\or mei\or juni\or  juli\or
 augustus\or  september\or  oktober\or november\or  december\or\fi,
\number\year}
\def\today{\ifcase\month\or
 Jan\or Febr\or  Mar\or  Apr\or May\or Jun\or  Jul\or
 Aug\or  Sep\or  Oct\or Nov\or  Dec\or\fi
 \space\number\day, \number\year}

\begin{document}

\newtheorem{theorem}{Theorem}[section]
\newtheorem{lemma}[theorem]{Lemma}
\newtheorem{proposition}[theorem]{Proposition}
\newtheorem{corollary}[theorem]{Corollary}
\newtheorem{conjecture}[theorem]{Conjecture}
\newtheorem{definition-lemma}[theorem]{Definition-Lemma}
\newtheorem{claim}[theorem]{Claim}
\theoremstyle{definition}
\newtheorem{definition}[theorem]{\bf Definition}
\newtheorem{example}[theorem]{\bf Example}
\newtheorem{none}[theorem]{}
\theoremstyle{remark}
\newtheorem{remark}[theorem]{\bf Remark}
\newtheorem{remarks}[theorem]{\bf Remarks}
\newtheorem{question}[theorem]{\bf Question}
\newtheorem{problem}[theorem]{\bf Problem}
\newcommand{\CC}{\mathbb C}
\newcommand{\DD}{\mathbb D}
\newcommand{\EE}{\mathbb E}
\newcommand{\FF}{\mathbb F}
\newcommand{\GG}{\mathbb G}
\newcommand{\HH}{\mathbb H}
\newcommand{\LL}{\mathbb L}
\newcommand{\PP}{\mathbb P}
\newcommand{\QQ}{\mathbb Q}
\newcommand{\RR}{\mathbb R}
\newcommand{\VV}{\mathbb V}
\newcommand{\WW}{\mathbb W}
\newcommand{\ZZ}{\mathbb Z}
\newcommand{\A}{\mathcal A}
\newcommand{\R}{\mathcal R}
\newcommand{\cO}{\mathcal O}
\newcommand{\cX}{\mathcal X}

\newcommand{\cI}{\mathcal I}
\newcommand{\sA}{{\mathcal A}^{\ast}}
\newcommand{\bM}{\overline{\mathcal M}}
\newcommand{\bA}{\overline{\mathcal A}}
\newcommand{\sF}{\mathscr {F}}
\newcommand{\cH}{\mathcal {H}}
\newcommand{\sH}{\mathscr {H}}
\newcommand{\sL}{\mathscr {L}}
\newcommand{\sZ}{\mathscr {Z}}

\newcommand\Sp{\operatorname{Sp}}
\newcommand\ch{\operatorname{ch}}
\newcommand\Fr{\operatorname{Fr}}
\newcommand\spec{\operatorname{Spec}}
\newcommand\an{\mathrm{an}}
\newcommand\gp{\mathrm{gp}}
\title[]{The ring of modular forms of degree two \\
 in characteristic three}

\author{Gerard van der Geer}
\address{Korteweg-de Vries Instituut, Universiteit van Amsterdam, Science 
Park 904, 1098 XH Amsterdam, The Netherlands, and
Universit\'e du Luxembourg, Unit\'e de Recherche en Math\'ematiques,
L-4364 Esch-sur-Alzette, Luxembourg.}
\email{g.b.m.vandergeer@uva.nl}


\maketitle
\begin{abstract}
We determine the structure of the ring of Siegel modular forms 
of degree $2$ in characteristic $3$.
\end{abstract}


\begin{section}{Introduction}
Let ${\A}_g$ be the moduli space of principally polarized abelian varieties of dimension $g$.
It is a Deligne-Mumford stack over ${\ZZ}$. 
It carries a natural vector bundle of rank~$g$, the Hodge bundle
${\EE}_g$. We write $L$ for its determinant line bundle. The vector bundle ${\EE}_g$ extends in a
natural way over any compactification $\tilde{\A}_g$ of Faltings-Chai type and we will
denote the extension of ${\EE}_g$ and $L$ again by the same symbols. 
Sections of $L^{\otimes k}$ over $\tilde{\A}_g$ are called modular forms of weight $k$.
It is known that for $g\geq 2$ any section of $L^k$ over ${\A}_g$ extends to a section of
$L^k$ over $\tilde{\A}_g$, a fact usually 
referred to as the Koecher principle, see \cite[Prop.\ 1.5, p.\ 140]{F-C}. 

If ${\FF}={\ZZ}$ or ${\ZZ}_p$ or a field one has the graded ring
$$
{\mathcal R}_g({\FF})= \oplus_k H^0(\tilde{\A}_g \otimes {\FF},L^k)\, .
$$
It is known by \cite{F-C} that it is a finitely generated ${\FF}$-algebra. 

In the case of ${\FF}={\CC}$ the ring ${\mathcal R}_g({\CC})$ is the ring
of scalar-valued Siegel modular forms of degree~$g$. It is well-known known that
${\mathcal R}_1({\CC})={\CC}[E_4,E_6]$ is freely generated over ${\CC}$ by the Eisenstein series $E_4$
and $E_6$ of weights $4$ and $6$. In the 1960s 
Igusa \cite{Igusa} determined the structure of ${\mathcal R}_2({\CC})$:
$$
{\mathcal R}_2({\CC})={\CC}[\psi_4,\psi_6,\chi_{10},\chi_{12}, \chi_{35}]/(\chi_{35}^2-P),
$$
where the indices of the generators indicate the weights and 
$P$ is a polynomial in $\psi_4,\psi_6,\chi_{10}$ and $\chi_{12}$.
Moreover, the ideal of cusp forms is generated by $\chi_{10},\chi_{12}$ and $\chi_{35}$.
For $g=3$, Tsuyumine showed in \cite{Tsuyumine} that ${\mathcal R}_3({\CC})$ is generated by $34$ 
elements; recently the number of generators
was reduced to $19$ by Lercier and Ritzenthaler \cite{L-R}.

For ${\FF}={\FF}_p$, a finite field with $p$ elements, the ring $R_1({\FF}_p)$
was described by Deligne \cite{Deligne1975}. Besides giving the structure
of the ring over ${\ZZ}$
$$
\mathcal{R}_1({\ZZ})={\ZZ}[c_4,c_6,\Delta]/(c_4^3-c_6^2-1728 \, \Delta) \, ,
$$
he showed that
$$
{\mathcal R}_1({\FF}_2)={\FF}_2[a_1,\Delta] \quad \text{\rm and} \quad 
{\mathcal R}_1({\FF}_3)={\FF}_3[b_2,\Delta]\, ,
$$
where $\Delta$ is of weight $12$ and $a_1$ (resp $b_2$) is of weight $1$ (resp.\ $2$).
For $p\geq 5$ we have ${\mathcal R}_1({\FF}_p)={\FF}_p[c_4,c_6]$.

For $g=2$,  Igusa determined in \cite{Igusa1979} 
also the ring of modular forms over ${\ZZ}$; it is generated by elements of weight
$$
4, 6, 10, 12, 12, 16, 18, 24, 28, 30, 35, 36, 40, 42, 48\, .
$$

For finite fields the structure of ${\mathcal R}_2({\FF}_p)$ is known for $p\geq 5$.
For this we refer
to Ichikawa's paper \cite{Ichikawa}.
For $p\geq 5$ the ring is just as in characteristic zero generated by modular forms $\psi_4$,
$\psi_6$, $\chi_{10}$, $\chi_{12}$ and $\chi_{35}$ with $\chi_{35}$ satisfying a relation
$\chi_{35}^2=P(\psi_4,\psi_6,\chi_{10},\chi_{12})$. Moreover for $p\geq 5$ the reduction map 
${\mathcal R}_2({\ZZ}_p) \to {\mathcal R}_2({\FF}_p)$ is surjective. 
Nagaoka studied the image of the reduction map in \cite{N}, see also \cite{B-N}.

\bigskip

In this paper we consider the case $p=3$ and determine the structure of ${\mathcal R}_2({\FF}_3)$.
We use the close connection between the moduli space ${\A}_2$ and the moduli space
${\mathcal M}_2$ of curves of genus $2$ via the Torelli map ${\mathcal M}_2 \hookrightarrow {\A}_2$
and the description of ${\mathcal M}_2$  as a quotient stack 
for the action of ${\rm GL}(2)$ on the space of binary sextics.
In that way invariant theory can be used to construct modular forms.
The relation between invariants and modular forms was already exploited by Igusa in
\cite{Igusa}, but he used theta functions and Thomae's formula to relate these to cross ratios
of the zeros of a binary sextic. 
Here we use not only invariants but also covariants giving vector-valued modular forms as
introduced in \cite{CFG-MathAnn} to analyze the regularity of scalar-valued modular forms. 

Our result is:

\begin{theorem}
The subring ${\mathcal R}_2^{\rm ev}({\FF}_3)$ of modular forms of even weight
 is generated by forms of weights $2,10,12,14$ and $36$ and has the form
$$
{\mathcal R}_2^{\rm ev}({\FF}_3)= {\FF}_3[\psi_2, \chi_{10}, \psi_{12}, \chi_{14},\chi_{36}]/J
$$
with $J$ the ideal generated by the relation 
$$
\psi_2^3\chi_{36}-\chi_{10}^3\psi_{12}-\psi_2^2\chi_{10}\chi_{14}^2+\chi_{14}^3\, .
$$
Moreover, ${\mathcal R}_2({\FF}_3)={\mathcal R}_2^{\rm ev}({\FF}_3)[\chi_{35}]/(\chi_{35}^2-P)$
with $P$ a polynomial in $\psi_2$, $\chi_{10}$, $\psi_{12},\chi_{14}$ and $\chi_{36}$.
The ideal of cusp forms is generated by $\chi_{10}, \chi_{14},\chi_{35}, \chi_{36}$.
\end{theorem} 

The generator $\psi_2$ is the Hasse invariant that vanishes on the locus of
non-ordinary abelian surfaces and $\chi_{10}$ is a form that 
vanishes on the locus of
products of elliptic curves. The ring of modular forms of degree $2$ in
characteristic $2$ is described in \cite{C-vdG}.
\end{section}
\begin{section}{The proof of Theorem 1.1.}
Since for $g=2$ the moduli stack ${\A}_g\otimes {\FF}_3$ has a canonical 
compactification due to Igusa we will use this compactification 
$\tilde{\A}_2 \otimes {\FF}_3$. We will denote the space of sections of $L^k$ on $\tilde{\A}_2 \otimes {\FF}_3$ by
$M_k(\Gamma_2)$ and we thus have ${\mathcal R}_2({\FF}_3)=\oplus_k M_k(\Gamma_2)$. 
We write $M_k(\Gamma_1)$ for the space $H^0(\tilde{\A}_1\otimes{\FF}_3,L^k)$.
The Satake compactification is denoted by ${\A}_2^{*}\otimes {\FF}_3$.
We denote the first Chern class of $L$ by $\lambda_1$.

\smallskip

We begin by constructing generators of weight $2$ and $10$.
The locus $V_1$ of abelian surfaces with $p$-rank $\leq 1$ is a divisor in ${\A}_{2}\otimes {\FF}_p$
and its closure 
$\overline{V}_1$ in $\tilde{\A}_2\otimes {\FF}_p$
has cycle class $(p-1)\lambda_1$ in the Chow group with ${\QQ}$-coefficients, 
so $[\overline{V}_1]=2\lambda_1$ for $p=3$, 
see \cite{vdG,E-vdG}. Therefore the effective divisor $\overline{V}_1$ is the
divisor of a section of $L^{\otimes 2}$ and
there is a modular form
$\psi_2$ of weight $2$ whose zero divisor is $\overline{V}_1$. It is determined up to multiplication
by a non-zero scalar. We will normalize it later. This form is known as
the Hasse invariant.  Multiplication by $\psi_2$  implies that
$\dim M_k(\Gamma_2) \leq \dim M_{k+2}(\Gamma_2)$. 

The divisor of products of elliptic curves $H_1:={\A}_{1,1}\otimes {\FF}_3$ gives rise to a second modular form. 
(The notation refers to the fact that $H_1$ is the Humbert surface of discriminant $1$.)
In the Chow group of codimension $1$ of $\tilde{\A}_2\otimes {\FF}_3$ (resp.\ ${\A}_2^*\otimes {\FF}_3$) 
we have the relation (cf.\ e.g. \cite[p.\ 317]{Mumford})
$$
2[\overline{H}_1]+[D]= 10\, \lambda_1 \quad ({\rm resp.\ } \quad 2[\overline{H}_{1}]= 10\, \lambda_1)\, ,
$$
with $D$ the divisor at infinity,
hence there exists a modular form of weight $10$ vanishing with multiplicity $2$ 
on ${H}_{1}$. We call this form  $\chi_{10}$
(up to a normalization to be determined later). 
The automorphism group of a generic product of
elliptic curves has an extra involution (when compared with
 the automorphism group of a generic principally
polarized abelian surface) and it acts by $-1$ on $L$,
hence every modular form of even weight
vanishes with even multiplicity along ${H}_{1}$.

Restriction to ${H}_{1}$ yields for even $k$ an exact sequence
$$
0 \to H^0(\mathcal{A}_2\otimes {\FF}_3,L^k\otimes O(-2H_1))\to 
H^0(\mathcal{A}_2\otimes {\FF}_3, L^k) \to H^0(H_1,L^k_{|H_1})
$$
and in view of the degree $2$ morphism $\mathcal{A}_1\times \mathcal{A}_1 \to
\mathcal{A}_{1,1}$ induced by interchanging the two factors, 
we can identify this with
$$
0 \to M_{k-10}(\Gamma_2) \to M_k(\Gamma_2) \to {\rm Sym}^2(M_k(\Gamma_1))\, ,
\eqno(1)
$$
where the second arrow is multiplication by $\chi_{10}$.
Moreover $M_{k-10}(\Gamma_2)=(0)$ for  $k <8$ since $L$ is ample
on ${\A}_2^*\otimes {\FF}_3$.
The exact sequence (1) and the fact that we know $M_k(\Gamma_1)$ implies that 
$\dim M_k(\Gamma_2)=1$  for $k=2,4,6,8$ and $\dim M_{10}(\Gamma_2)=2$ and 
$M_{10}(\Gamma_2)$ is
generated by $\psi_2^5$ and $\chi_{10}$.
\bigskip

We now turn to the construction of the other generators. We use the ideas of \cite{CFG-MathAnn}.
The Torelli map defines an embedding ${\mathcal M}_2\otimes {\FF}_3 \to {\A}_2\otimes {\FF}_3$. 
A smooth projective curve of genus $2$ can be given by an equation
$$
y^2=f(x) \quad \text{\rm with $f=\sum_{i=0}^6 a_i \, x^{6-i}$} \, .
\eqno(3)
$$
We let $V=\langle x_1,x_2 \rangle$ be the  ${\FF}_3$-vector space
generated by $x_1,x_2$
and write $f$ as a homogeneous polynomial $\sum_{i=0}^6 a_i x_1^{6-i}x_2^i$.
Note that a curve as in (3) comes with a basis of the space of regular
differentials, viz.\  $dx/y, \, xdx/y$.

We have a  description of ${\mathcal M}_2\otimes {\FF}_3$ 
as the stack quotient $[\mathcal{X}^0/{\rm GL}(V)]$
with $\mathcal{X}^0 \subset \mathcal{X}={\rm Sym}^6(V) \otimes \det(V)^{-2}$
the locus given by the non-vanishing of the discriminant, see
\cite[Section 3, p.\ 3]{CFG-MathComp}.

The pullback to ${\mathcal X}^0$ of the Hodge bundle 
under the composition of ${\mathcal X}^0 \to {\mathcal M}_2$
with the Torelli map ${\mathcal M}_2 \hookrightarrow {\A}_2$
is the equivariant bundle $V$ on ${\mathcal X}^0$ as the basis  
$dx/y, \, xdx/y$
of the space of regular differentials on the curve $y^2=f(x)$ shows.
The pullback of $L$ is $\det(V)$. As a consequence pulling back defines a homomorphism
$$
\mu: {\mathcal R}_2({\FF}_3) \to I \eqno(4)
$$
with $I$ the ring of invariants of the action of ${\rm GL}(V)$ on ${\rm Sym}^6(V)$.
Here an invariant is a polynomial in $a_0,\ldots,a_6$, the coefficients of $f$
that is invariant under ${\rm SL}(V)$. Since the image of ${\mathcal M}_2$ in ${\A}_2$
is a Zariski open part with complement $H_1$, not every invariant corresponds to a modular form; but
every invariant corresponds to a rational modular form that is regular 
outside $H_{1}$. In particular, it becomes regular on all of ${\A}_2$
when multiplied with a sufficiently high power of $\chi_{10}$. This provides us with 
homomorphisms
$$
{\mathcal R}_2({\FF}_3) {\buildrel \mu \over \longrightarrow}  I 
{\buildrel \nu \over \longrightarrow} {\mathcal R}_2({\FF}_3)_{\chi_{10}}\, ,
$$
where ${\mathcal R}_2({\FF}_3)_{\chi_{10}}$ is obtained from ${\mathcal R}_2({\FF}_3)$
by allowing powers of $\chi_{10}$ in the denominator. We have $\nu \circ \mu={\rm id}$.

\smallskip
This generalizes as follows to vector-valued modular forms.
For each finite dimensional irreducible representation $\rho$ of ${\rm GL}(2)$ there is
a vector bundle ${\EE}_2^{\rho}$ obtained from ${\EE}_2$ by applying a Schur functor.
Such a $\rho$ is of the form ${\rm Sym}^j({\rm St}) \otimes \det^k({\rm St})$
with ${\rm St}$ the standard representation of ${\rm GL}(V)$. A section of
${\rm Sym}^j ({\EE}_2) \otimes \det({\EE}_2)^k$ over ${\A}_2$
is called a modular form of degree $2$
and weight $(j,k)$. The Koecher principle also applies to these modular forms:
sections of ${\EE}_2^{\rho}$ over ${\A}_2$ extend over $\tilde{\A}_2$, see
\cite[Prop.\ 1.5, p.\ 140]{F-C}.
We write 
$$
M_{j,k}(\Gamma_2)=H^0(\tilde{\A}_2\otimes {\FF}_3, {\rm Sym}^j({\EE}_2) \otimes \det({\EE}_2)^k)
$$
and we consider the ${\mathcal R}_2({\FF}_3)$-module
$$
M= \oplus_{j,k} M_{j,k}(\Gamma_2)\, .
$$
It is even a ring. The map (4) can be extended to a map from $M$ to the ring of
covariants. Here a covariant can be described as an invariant for the action of ${\rm GL}(V)$
on $V \oplus {\rm Sym}^6(V)$. Alternatively, covariants can be obtained by taking an
equivariant embedding of an irreducible ${\rm GL}(V)$-representation 
$U \to {\rm Sym}^d({\rm Sym}^6(V))$,
or equivalently, an equivariant map 
$$
\varphi: {\FF}_3 \to {\rm Sym}^d({\rm Sym}^6(V))\otimes U^{\vee}
$$
and then $\Phi=\varphi(1)$ is a covariant. If $U$ is an irreducible representation of
highest weight $(w_1,w_2)$ then one may view $\Phi$ as a homogeneous 
form in $a_0,\ldots,a_6$ of degree $d$
and in $x_1, x_2$ of degree $w_1-w_2$, see \cite{Springer, G-Y} and 
\cite{CFG-MathAnn}. For example, taking $U={\rm Sym}^6(V)$ and $d=1$
yields the covariant $\Phi=f$, the universal binary sextic. 
Covariants form a ring ${\mathcal C}$ that was much studied in the 19th and early 20th century. 
Grace and Young determined generators of this ring in \cite{G-Y}.

The maps ${\mathcal R}_2({\FF}_3) \to I \to {\mathcal R}_2({\FF}_3)_{\chi_{10}}$ now extend to
$$
M {\buildrel \mu \over \longrightarrow }\, 
{\mathcal C} {\buildrel \nu \over \longrightarrow} M_{\chi_{10}}\, ,
$$
where $M_{\chi_{10}}$ is obtained from $M$ by admitting powers of $\chi_{10}$ as denominators.
We have $\nu \circ \mu= {\rm id}_M$.

\bigskip

The image under $\nu$ of the covariant $f$, the universal binary sextic, 
 is a rational modular
form $\chi_{6,-2}$, that is, a rational section of ${\rm Sym}^6({\EE}_2) \otimes {\det}({\EE}_2)^{-2}$
that is regular after multiplication by an appropriate  power of $\chi_{10}$. 
The power $-2$ comes from the twisting used in the description of the
stack quotient $[\mathcal{X}^0/{\rm GL}(V)]$, where $\mathcal{X}^0 \subset
{\rm Sym}^6(V) \otimes \det(V)^{-2}$, see  \cite[Section 3, p.\ 3]{CFG-MathComp}.

This construction was given in \cite{CFG-MathAnn} in characteristic zero and yields a
meromorphic modular form, here denoted $\varphi_{6,-2}$, 
that becomes holomorphic after multiplication by $\chi_{10}$.
The reduction of the characteristic zero rational modular form $\varphi_{6,-2}$ yields a
rational modular form in characteristic $3$. 
This implies that $\chi_{6,-2}$ becomes
regular after multiplication by $\chi_{10}$. We can write the form $\chi_{6,-2}$ locally
on ${\A}_2\otimes {\FF}_3$
symbolically as 
$$
\chi_{6,-2}= \sum_{i=0}^6 \alpha_i X_1^{6-i} X_2^i\, , \eqno(5)
$$
where the monomials $X_1^{6-i} X_2^i$ are dummies to indicate the coordinates in the fibres
of ${\rm Sym}^6({\EE}_2) \otimes {\det}({\EE}_2)^{-2}$.  Here we view $\alpha_i$ locally as 
a rational function on ${\A}_2 \otimes {\FF}_3$. 
Using the local expression (5) one can give the image $\nu(T)$ of an invariant 
$T=T(a_0,\ldots,a_6)$ locally by
$T(\alpha_0,\ldots,\alpha_6)$.

We note that interchanging $X_1$ and $X_2$ induces an involution replacing $\alpha_i$
by $\alpha_{6-i}$. 

Comparing with the characteristic $0$ case and using semi-continuity
we see that the orders of the rational functions $\alpha_i$ along the divisor $H_1$
are at least equal to the orders of their complex analogues along $H_1$. The Fourier 
expansion in characteristic $0$ given in \cite[page 1658]{CFG-MathAnn} implies the following inequalities
for the orders of $\alpha_i$ along $H_1$ in characteristic $3$:
$$
{\rm ord}_{H_{1}}(\alpha_0,\ldots,\alpha_6)=
(\geq 2, \geq 1, \geq 0, \geq -1, \geq 0, \geq 1, \geq 2)\, . \eqno(6)
$$
Moreover, the symmetry that interchanges $x_1$ and $x_2$ implies that the orders
of $\alpha_i$ and $\alpha_{6-i}$ along $H_1$ are equal.
Another way to see the estimates for the orders is by developing $\chi_{6,8}=
\chi_{6,-2} \chi_{10}$ along the locus ${\mathcal A}_{1,1} \otimes {\FF}_3 \subset
 {\mathcal A}_2\otimes {\FF}_3$. Since the pullback of the Hodge bundle ${\EE}_2$
to ${\mathcal A}_1 \times {\mathcal A}_1$ via ${\mathcal A}_1^2 \to {\mathcal A}_{1,1} 
\subset {\mathcal A}_2$ is $\oplus_{i=0}^6 p_1^*({\EE}_1)^i \otimes p_2^*({\EE}_1)^{6-i}$
the restriction of $\alpha_i \chi_{10}$ lies in $S_{14-i}(\Gamma_1)\otimes S_{8+i}(\Gamma_1)$ and this is zero. The next Taylor term in the Taylor development along 
${\mathcal A}_{1,1}$ lies in $S_{15-i}(\Gamma_1) \otimes S_{9+i}(\Gamma_1)$ and this is zero for $i\neq 3$.
\bigskip

The ring of invariants $I$ for the action of ${\rm GL}(V)$ on ${\rm Sym}^6(V)$ in
characteristic $3$ is generated by invariants $A$, $B$, $C$, $D$ and $E$ of degree
$2$, $4$, $6$, $10$ and $15$, see e.g.\ \cite{Igusa} or \cite{Geyer}.
The invariants $A,B,C,D$ that we use here can be expressed 
in the reductions modulo $3$
of the invariants $J_2,J_4,J_6$ et $J_{10}$ given in \cite{Liu}:
$A=-J_2 (\bmod 3)$, $B= -J_4 (\bmod 3)$, $C=-J_6-A^3 (\bmod 3)$, $D= J_{10} (\bmod 3)$. The invariant $E$ can be found in \cite[p.\ 848]{Igusa1967}.

The invariant $A$ has the form $A=a_1a_5-a_2a_4$. We know of the existence of a
modular form $\psi_2$ of weight $2$. Under the map $\mu$ it must map to a non-zero
multiple of $A$. 
We fix $\psi_2$ by requiring $\mu(\psi_2)=A$.
The restriction to $H_1$ of the Hasse invariant $\psi_2$ is a non-zero multiple
of ${\rm Sym}^2(b_2)$, with $b_2$ the Hasse invariant for $g=1$, 
hence $\psi_2$ does not vanish identically on $H_{1}$. 

By the inequalities (6) and the expression for $A$ 
we see that ${\rm ord}_{{H}_{1}}(\alpha_2)=0=
{\rm ord}_{{H}_{1}}(\alpha_4)$  and
$$
{\rm ord}_{H_{1}}(\alpha_0,\ldots,\alpha_6)=
(\geq 2, \geq 1, 0, \geq -1, 0, \geq 1, \geq 2)\, .
$$

In degree $4$ we find another invariant $B$, not a multiple of $A^2$:
$$
\begin{aligned}
B=2\, a_0a_1a_5a_6+a_0a_2a_4a_6+2\, a_0a_2a_5^2+2\, a_0a_4^3+
2\, a_1^2a_4a_6+2\, a_1a_2a_4a_5+ & \cr
a_1a_3^2a_5+a_1a_3a_4^2+2\,a_2^3a_6+a_2^2a_3a_5+a_2^2\, a_4^2+2\, a_2a_3^2a_4
& . \cr
\end{aligned}
$$
Since we know $\dim M_4(\Gamma_2)=1$ there cannot be a regular
modular form in weight $4$ that is not a multiple of $\psi_2^2$.
This implies that ${\rm ord}_{H_1}(\alpha_3)<0$ and hence ${\rm ord}_{H_1}(\alpha_3)=-1$.
Thus $B=(a_1a_5-a_2a_4) a_3^2+ (a_1a_4^2+a_2^2a_5) a_3+ \cdots$
defines a rational modular form $\chi_B=\nu(B)$ of weight $4$
with order $-2$ along $H_1$.
Since $\chi_{10}$ vanishes with multiplicity $2$ along $H_1$ we thus find that  
$$\chi_{14}:=\chi_B \chi_{10}
$$ 
is a regular modular form of weight $14$.

The vector space of invariants of degree $6$ is generated by $A^3$, $AB$ and an invariant $C$
$$
C= 2 \, a_3^6 + A \, a_3^4+ 2(a_1a_4^2+a_2^2a_5)a_3^3+ \cdots
$$
and we see that $\chi_C=\nu(C)$ has order $-6$ along $H_1$.
In degree $10$ there is a new invariant
$$
D=(a_1a_5)^3a_3^4+(a_0a_2^3a_5^3+a_1^3a_4^3a_6+ 2 a_1^3a_4^2a_5^2+2a_1^2a_2^2a_5^3)a_3^3+\cdots
$$
yielding a modular form that vanishes with multiplicity $\geq 2$ on ${H}_{1}$.
Indeed, since $\alpha_1\alpha_5$ vanishes with multiplicity $\geq 2$
the first term $(\alpha_1\alpha_5)^3 \alpha_3^4 $ 
vanishes with order $\geq 2$; the next terms also vanish with order $\geq 2$
as one easily checks.
Therefore $\chi_D$ is regular and vanishes with multiplicity $\geq 2$. Since $\chi_D$ is not zero, 
it must be a multiple of $\chi_{10}$ and then vanishes on $H_1$ with 
multiplicity~$2$. 
This implies that the order of vanishing of $\alpha_1$ and $\alpha_5$ 
along $H_1$ is $1$.

\begin{corollary} We have
$ {\rm ord}_{H_1}(\alpha_0,\ldots,\alpha_6)=(\geq 2, 1, 0, -1, 0, 1, \geq 2)$.
\end{corollary} 

We fix $\chi_{10}$ by setting it equal to $\chi_D=\nu(D)$. This fixes $\chi_{14}$ too.

In a similar manner one checks that the rational modular form $\psi_S=\nu(S)$
with $S$ equal to
$$
S= B^3+A^3C-A^2B^2=(a_1a_4^2+a_2^2a_5)^3a_3^3+ \cdots
$$
is regular too. We put $\psi_{12}=\psi_S$. We thus find a $3$-dimensional
subspace of $M_{12}(\Gamma_2)$ generated by $\psi_2^6, \psi_2 \chi_{10}$ and
$\psi_{12}$. From the fact that $B$ and $D$ are not divisible by $A$ 
we see that $\chi_{14}$ does not lie in $\psi_2 M_{12}(\Gamma_2)$.
Therefore $\dim M_{12}(\Gamma_2) < \dim M_{14}(\Gamma_2)$. Since we know by (1) that
$\dim M_{14}(\Gamma_2)\leq 4$ we conclude that $\dim M_{12}(\Gamma_2)=3$.

A further generator is 
$$
\chi_{36}=\nu(C D^3)=\chi_{C}\chi_{10}^3\, .
$$
Since the orders of $\chi_C$ and $\chi_{10}$ along $H_1$
are $-6$ and $2$ the modular form
$ \chi_{36}$
is regular and does not vanish identically on $H_1$. 
The modular form $\chi_{36}$ is not contained in the subring generated
by $\psi_2$,$\chi_{10}$, $\psi_{12}$ and $\chi_{14}$ as one sees by looking at
the invariants. 
We have the identity
$$
(B^3+A^3C-A^2B^2) D^3= B^3 D^3 + A^3C D^3 -A^2DB^2D^2
$$
by which we can express $\psi_{12}\chi_{10}^3$ in the other generators:
$$
\psi_{12}\chi_{10}^3=\chi_{14}^3+\psi_2^3\chi_{36}-\psi_2^2\chi_{10} \chi_{14}^2\, . \eqno(6)
$$
Since $A,B,C,D$ are generators of the ring of invariants and are algebraically
independent the forms $\psi_2, \chi_{10}, \psi_{12}, \chi_{14}$ are algebraically independent. The form $\chi_{36}$ then satisfies the algebraic relation (6)
and since there is no non-trivial 
relation of lower weight involving $\chi_{36}$ it
implies that this relation generates
the ideal of  relations between the generators $\psi_2$, $\chi_{10}$,
$\psi_{12}$, $\chi_{14}$ and $\chi_{36}$.

The forms  $\psi_2$, $\chi_{10}$,
$\psi_{12}$, $\chi_{14}$ and $\chi_{36}$ generate a subring $R^{\rm ev}$ 
of the ring
${\mathcal R}_2^{\rm ev}({\FF}_3)$ with generating function
$$
G=\frac{(1-t^{42})}{(1-t^2)(1-t^{10})(1-t^{12})(1-t^{14})(1-t^{36})}\, .
$$
and by the Riemann-Roch theorem we have
$\dim M_{k}(\Gamma_2)=k^3/1080 + O(k^2)$ for even $k$. 
Note that
$$
\frac{42}{2\cdot 10 \cdot 12 \cdot 14\cdot 36}= \frac{1}{2880}\, .
$$
On the other hand we have $c_1(L)^3=1/2880$, see \cite[p.\ 74]{vdG}. 
We can use the degree of ${\rm Proj}({\mathcal R}_2^{\rm ev}({\FF}_3))$
to show that there cannot be more generators
of ${\mathcal R}_2^{\rm ev}({\FF}_3)$, but 
one can see this also in a more elementary way as follows.

Let $d(k)=\dim M_{k}(\Gamma_2)$ and $r(k)=\dim R_k$ where 
$R_k= R^{\rm ev} \cap M_k(\Gamma_2)$.
\begin{proposition} We have $d(k)=r(k)$ for even $k\geq 0$.
\end{proposition}
\begin{proof} We know that $d(k) \geq r(k)$ for even $k$ and $d(k)=r(k)$ for
even $0\leq k \leq 14$. Suppose by induction that $d(k) =r(k)$ for even $k \leq m$.
The exact sequence (1) gives the upper bound
$d(k) \leq r(k-10)+ c(k)(c(k)+1)/2$ for $k\leq m+10$, where
$c(k)=\dim M_k(\Gamma_1)= \lfloor k/12\rfloor +1$.
Using the generating function $G$ one sees that
$r(k)-r(k-10)= c(k)(c(k)+1)/2$ for $k\not\equiv 0 \, (\bmod \, 12)$ and
$k \not\equiv  2 \, (\bmod \, 12)$. Hence $d(k)=r(k)$ for even $k\leq m+10$
with $k \not\equiv 0, 2 (\bmod 12)$. But we have
$$
d(k+2)-d(k) \geq r(k+2)-r(k)\, ,
$$
as we show in the next lemma. This proves $d(k)=r(k)$ for even $k \leq m+10$.
Therefore we conclude the proof by induction.
\end{proof}

\begin{lemma}
We have $d(k+2)-d(k) \geq r(k+2)-r(k)$ for even $k\geq 0$.
\end{lemma}
\begin{proof}
We can write $R_{k+2}=\psi_2 R_k \oplus N_{k+2}$ with $N_{k+2}$ the subspace
with basis the forms $\chi_{10}^a \psi_{12}^b \chi_{14}^c \chi_{36}^d$
with $a,b,c,d \geq 0$ and $c\leq 2$ in view of the relation (6).
Then we have $\dim N_{k+2}=r(k+2)-r(k)$. The inequality
$d(k+2)-d(k)\geq \dim N_{k+2}$ follows from the fact
that $N_{k+2}\cap \psi_2 M_k(\Gamma_2) =(0)$. To see this fact,
suppose that $f \in M_{k}(\Gamma_2)$ such that $f \not\in R_k$ and
$\psi_2 f \in R_{k+2}$. Then $\psi_2 f =P$ with $P$ a sum of monomials
$\chi_{10}^a \psi_{12}^b \chi_{14}^c \chi_{36}^d$ with $c \leq 2$.
Then $P=\nu(Q)$ with $Q$ a polynomial in
$$
D, B^3+A^3C-A^2B^2, BD, CD^3 \, .
$$ 
Since $P=\psi_2 f$ this 
polynomial must be divisible by $A$. But this implies that if $Q\neq 0$ then
it must have at least one monomial with $c\geq 3$, but we excluded this.
\end{proof}

The invariant $E$ of degree $15$ is of the form
$$
E= (a_1a_4^2-a_2^2a_5)^3 a_3^6+\cdots
$$
and $\nu(E)$ has order $-3$ along $H_1$. Therefore
$$
\chi_{35}:=\nu(E D^2)
$$
is a regular modular form. It vanishes on $H_1$ and on the Humbert surface $H_4$ of
discriminant $4$, both with multiplicity $1$. The surfaces $H_1$ and $H_4$ parametrize
abelian surfaces that possess an extra involution.
Locally near $H_4$ the extra automorphism
corresponds to the symmetry that interchanges $x_1$ and $x_2$.

We know that the cycle class of $2\, H_4$ on ${\A}_2^*\otimes {\FF}_3$
is $60 \lambda_1$, see \cite[Prop.\ 3.3, p.\ 217]{vdG-HMS}. 
Therefore the divisor of $\chi_{35}$ is $H_1+H_4$ and since the closure of $H_1$
contains the $1$-dimensional cusp  
$\chi_{35}$ is a cusp
form.
Then $\chi_{35}^2$ is of even weight, hence can be expressed as a 
polynomial in $\psi_2,\chi_{10},\psi_{12}, \chi_{14}$ and $\chi_{36}$. 
If $\psi$ is an odd weight
modular form then it must vanish on $H_1$ and $H_4$, hence it will be divisible by
$\chi_{35}$.

The relation between the space of binary sextics and the moduli space
$\overline{\mathcal{M}}_2$ (see for example \cite[Section 4]{CFG-MathComp})
implies  that a modular form $\chi$ is a cusp form if and only if
the invariant $\mu(\chi)$ 
is divisible by the discriminant $D$ in $I$. 
From the form of the generators one easily sees that $\chi_{10},\chi_{14}, \chi_{36}$ and
$\chi_{35}$  generate the 
ideal of cusp forms. This completes the proof.

\bigskip

\begin{remark}
One can use the knowledge of the dimensions of $M_k(\Gamma_2)$ to deduce non-vanishing
of $H^1(\tilde{\A}_2\otimes {\FF}_3, L^k)$ for certain values of $k$. 
The short exact sequence of sheaves on $\tilde{\A}_2\otimes {\FF}_3$
$$
0 \to L^k\otimes {\mathcal O}(-\overline{V}_1) \to L^k \to L^k_{|\overline{V}_1} \to 0
$$
gives rise to a long exact sequence 
which can be identified with
$$
0 \to M_{k-2}(\Gamma_2) \to M_k(\Gamma_2) \to H^0(\overline{V}_1,L^k) 
\to H^1(\tilde{\A}_2\otimes {\FF}_3,L^{k-2}) \to \cdots
$$ 
For example, if $\dim M_{k-2}(\Gamma_2)=\dim M_k(\Gamma_2)$ we get an injection
$ H^0(\overline{V}_1,L^k) \to H^1(\tilde{\A}_2\otimes {\FF}_3,L^{k-2})$ and if $k\equiv 0 (\bmod \, 4)$
and $k\geq 0$ 
one can show that $H^0(\overline{V}_1,L^k)\neq 0$ by showing that 
 $H^0(\overline{V}_1[2],L^k)^{\mathfrak{S}_6}\neq (0)$, the space of invariants
under the symmetric group $\mathfrak{S}_6$ acting on $H^0(\overline{V}_1[2],L^k)$ with
$V_1[2]$ the $3$-rank $\leq 1$ locus in the level $2$ moduli space $\tilde{\A}_2[2]$.
Thus for example, $H^1(\tilde{\A}_2\otimes {\FF}_3,L^{14})\neq (0)$.
\end{remark}
\end{section}
\section*{Acknowledgements}
The author thanks Fabien Cl\'ery for helpful correspondence.
He thanks YMSC at Tsinghua University where part of this work was done for hospitality. Finally thanks are due to the referee for his/her remarks.

\end{document}